\documentclass[pdflatex,sn-mathphys-num]{sn-jnl}

\usepackage{graphicx}
\usepackage{amsmath,amssymb,amsfonts}
\usepackage{microtype}
\usepackage{booktabs}
\usepackage{array}
\usepackage{centernot}
\usepackage{enumitem}

\providecommand{\jyear}[1]{}
\jyear{2026}

\newtheorem{theorem}{Theorem}[section]
\newtheorem{proposition}[theorem]{Proposition}
\newtheorem{lemma}[theorem]{Lemma}
\newtheorem{corollary}[theorem]{Corollary}
\newtheorem{definition}[theorem]{Definition}

\theoremstyle{remark}

\numberwithin{equation}{section}

\newcommand{\Syl}{\operatorname{Syl}}

\newcommand{\Rad}{\operatorname{Rad}}

\newcommand{\lcm}{\operatorname{lcm}}

\newcommand{\C}{\mathsf C}

\newcommand{\Sp}{\mathsf{SP}}
\newcommand{\Syls}{\mathcal S}
\newcommand{\gammainv}{\gamma}
\newcommand{\nuu}{\nu}
\newcommand{\sig}{\sigma}

\title[Compensation theorem for the Sylow-integral invariant]{A compensation theorem for the Sylow-integral invariant and counterexamples to an \texorpdfstring{$A_5$}{A5}-characterization conjecture}

\author*[1]{\fnm{Yutong} \sur{Zhang}}\email{yutongzhang@stu.scu.edu.cn}
\author[1]{\fnm{Yaoran} \sur{Yang}}\email{yangyaoran@stu.scu.edu.cn}
\affil[1]{\orgdiv{School of Mathematics}, \orgname{Sichuan University}, \orgaddress{\city{Chengdu}, \postcode{610065}, \country{China}}}

\begin{document}

\abstract{Let \(\nu_p(G)\) be the number of Sylow \(p\)-subgroups of a finite group \(G\), let \(\sigma_p(G)\) be their common order, and set
\[
        \gamma(G)=\int_0^1\sum_{p\in\pi(G)}\nu_p(G)x^{\sigma_p(G)}\,dx
        =\sum_{p\in\pi(G)}\frac{\nu_p(G)}{\sigma_p(G)+1}.
\]
A recent conjectural extension of the simple-group theorem for this invariant asserted that a nonsolvable finite group has \(\gamma(G)=9/2\) precisely when \(G\cong A_5\).  We disprove this assertion by a direct and verifiable construction.  More generally, we prove an exact direct-product compensation formula for \(A_5\) with an arbitrary nilpotent factor.  The formula reduces the equality \(\gamma(A_5\times N)=9/2\) to a finite Egyptian-fraction equation in the orders of the Sylow subgroups of \(N\).  Taking \(N=\C_2\times\C_7\times\C_{11}\times\C_{13}\times\C_{17}\times\C_{19}\times\C_{29}\times\C_{71}\times\C_{83}\), the loss in the \(2\)-Sylow contribution is exactly compensated by the new normal Sylow subgroups.  Consequently \(G=A_5\times N\) is nonsolvable, is not isomorphic to \(A_5\), has solvable radical \(N\), and nevertheless satisfies \(\gamma(G)=9/2\).  Several further explicit compensation certificates are also recorded.}



\maketitle

\section{Introduction}

The present paper concerns a Sylow-counting invariant of finite groups.  For a finite group \(G\) and a prime \(p\), write
\[
        \pi(G)=\{p: p\mid |G|\},\qquad
        \Syls_p(G)=\{P\leq G:P\in\Syl_p(G)\}.
\]
All members of \(\Syls_p(G)\) have the same order.  We use the notation
\[
        \nuu_p(G)=|\Syls_p(G)|,\qquad
        \sig_p(G)=|P|\quad(P\in\Syls_p(G)).
\]
The Sylow polynomial and its integral are
\begin{equation}\label{eq:intro-sp}
        \Sp(G,x)=\sum_{p\in\pi(G)}\nuu_p(G)x^{\sig_p(G)},
        \qquad
        \gammainv(G)=\int_0^1\Sp(G,x)\,dx
        =\sum_{p\in\pi(G)}\frac{\nuu_p(G)}{\sig_p(G)+1}.
\end{equation}
This invariant is close in spirit to the recent finite-group polynomial invariants considered in \cite{AsboeiAnabantiRicerche}.  The simple-group case of \eqref{eq:intro-sp} was treated by Anabanti and Asboei \cite{AnabantiAsboeiSiberian}: if \(G\) is finite noncyclic simple, then \(\gamma(G)=9/2\) if and only if \(G\cong A_5\).  A public 2024 conjecture attributed to Asboei proposed the same characterization with ``finite nonsolvable group'' in place of ``finite noncyclic simple group'' \cite{AnabantiSlides}.  The conjectural assertion is therefore
\begin{equation}\label{eq:conjectural-statement}
        G\ \,\text{finite nonsolvable},\quad
        \gammainv(G)=\frac92
        \quad\Longleftrightarrow\quad
        G\cong A_5.
\end{equation}
Our purpose is to settle \eqref{eq:conjectural-statement} negatively, and to do so in a way that explains the exact obstruction to extending the simple-group result.

The mechanism is unexpectedly rigid.  If \(N\) is nilpotent and \(N_p\) denotes its Sylow \(p\)-subgroup, then the direct-product rule proved below gives
\begin{equation}\label{eq:intro-formula}
\begin{aligned}
        \gammainv(A_5\times N)
        &=\frac{5}{4|N_2|+1}+\frac{10}{3|N_3|+1}+\frac{6}{5|N_5|+1}
          +\sum_{q\in\pi(N)\setminus\{2,3,5\}}\frac{1}{|N_q|+1},
\end{aligned}
\end{equation}
where \(|N_p|=1\) when \(p\notin\pi(N)\).  Thus the old primes \(2,3,5\) decrease the three \(A_5\)-terms, while new primes add positive terms.  Equality with \(9/2\) is therefore a balancing problem, not a simplicity phenomenon.

The smallest old-prime defect is obtained by adjoining a central \(2\)-group of order \(2\):
\begin{equation}\label{eq:defect-c2}
        \frac{5}{4+1}-\frac{5}{4\cdot 2+1}=1-\frac59=\frac49.
\end{equation}
The following elementary identity supplies exactly this amount from new normal Sylow subgroups:
\begin{equation}\label{eq:key-identity-intro}
        \frac18+\frac1{12}+\frac1{14}+\frac1{18}+\frac1{20}+\frac1{30}+\frac1{72}+\frac1{84}=\frac49.
\end{equation}
Since
\[
\begin{gathered}
        8=7+1,\quad 12=11+1,\quad 14=13+1,\quad 18=17+1,\\
        20=19+1,\quad 30=29+1,\quad 72=71+1,\quad 84=83+1.
\end{gathered}
\]
\eqref{eq:defect-c2} and \eqref{eq:key-identity-intro} immediately suggest a counterexample.

\begin{theorem}\label{thm:main-counterexample}
Let
\begin{equation}\label{eq:m0-def}
        m_0=7\cdot 11\cdot 13\cdot 17\cdot 19\cdot 29\cdot 71\cdot 83
        =55254930731
\end{equation}
and let
\begin{equation}\label{eq:g0-def}
        G_0=A_5\times \C_2\times \C_7\times \C_{11}\times \C_{13}\times \C_{17}\times \C_{19}\times \C_{29}\times \C_{71}\times \C_{83}.
\end{equation}
Then \(G_0\) is nonsolvable, \(G_0\not\cong A_5\), and
\begin{equation}\label{eq:g0-gamma}
        \gammainv(G_0)=\frac92.
\end{equation}
Consequently the nonsolvable-group assertion \eqref{eq:conjectural-statement} is false.
\end{theorem}

The proof is not a reduction to a known theorem.  It is an exact computation using the behavior of Sylow numbers under products.  The Sylow polynomial of \(G_0\) is
\begin{equation}\label{eq:g0-polynomial}
\begin{aligned}
        \Sp(G_0,x)
        &=5x^8+10x^3+6x^5+x^7+x^{11}+x^{13}+x^{17}+x^{19}+x^{29}+x^{71}+x^{83},
\end{aligned}
\end{equation}
so that
\begin{equation}\label{eq:g0-integral-intro}
\begin{aligned}
        \gammainv(G_0)
        &=\frac59+\frac{10}{4}+\frac66
          +\frac18+\frac1{12}+\frac1{14}+\frac1{18}+\frac1{20}+\frac1{30}+\frac1{72}+\frac1{84} \\
        &=\frac59+\frac52+1+\frac49
         =\frac92.
\end{aligned}
\end{equation}
The remaining sections prove this computation in a structural form and record several consequences.

\section{Sylow numbers under direct products}

The direct-product formula is elementary, but it is the point at which the conjectural extension from simple groups to nonsolvable groups breaks down.  We give the complete statement because the invariant \(\gamma\) is neither additive nor multiplicative.

\begin{definition}\label{def:neutral-convention}
Let \(p\) be any prime.  If \(p\notin\pi(G)\), we use the neutral conventions
\begin{equation}\label{eq:neutral}
        \Syls_p(G)=\{1\},\qquad \nuu_p(G)=1,
        \qquad \sig_p(G)=1.
\end{equation}
When \(p\in\pi(G)\), the same symbols have their ordinary Sylow meanings.
\end{definition}

The convention \eqref{eq:neutral} is not meant to add a spurious term to \(\gamma(G)\).  It is only a compact way of writing product formulas.  Sums defining \(\gamma(G)\) are always taken over \(\pi(G)\).

\begin{lemma}[Direct-product rule]\label{lem:direct-product}
Let \(G_1,\ldots,G_r\) be finite groups and put
\[
        G=G_1\times\cdots\times G_r.
\]
For every prime \(p\),
\begin{equation}\label{eq:sylow-product-set}
        \Syls_p(G)=\{P_1\times\cdots\times P_r:
        P_i\in\Syls_p(G_i)\text{ for every }i\}.
\end{equation}
Consequently, with the neutral convention \eqref{eq:neutral},
\begin{equation}\label{eq:product-nu-sigma}
        \nuu_p(G)=\prod_{i=1}^r\nuu_p(G_i),
        \qquad
        \sig_p(G)=\prod_{i=1}^r\sig_p(G_i).
\end{equation}
Thus
\begin{equation}\label{eq:gamma-product-general}
        \gammainv(G_1\times\cdots\times G_r)
        =\sum_{p\in\pi(G_1)\cup\cdots\cup\pi(G_r)}
        \frac{\prod_{i=1}^r\nuu_p(G_i)}{1+\prod_{i=1}^r\sig_p(G_i)}.
\end{equation}
\end{lemma}

\begin{proof}
Write \(|G_i|=p^{a_i}m_i\), where \(p\nmid m_i\).  A subgroup
\[
        P=P_1\times\cdots\times P_r\leq G_1\times\cdots\times G_r
\]
with \(|P_i|=p^{a_i}\) has order \(p^{a_1+\cdots+a_r}\), which is the largest power of \(p\) dividing \(|G|\).  Hence every such product is Sylow in \(G\).  Conversely, if \(P\in\Syls_p(G)\), then the projection \(\operatorname{pr}_i(P)\) is a \(p\)-subgroup of \(G_i\), and
\[
        |P|\leq \prod_{i=1}^r |\operatorname{pr}_i(P)|\leq p^{a_1+\cdots+a_r}=|P|.
\]
Both inequalities are equalities.  Hence \(|\operatorname{pr}_i(P)|=p^{a_i}\) for every \(i\), so \(\operatorname{pr}_i(P)\in\Syls_p(G_i)\), and equality in the first inequality forces
\[
        P=\operatorname{pr}_1(P)\times\cdots\times \operatorname{pr}_r(P).
\]
This proves \eqref{eq:sylow-product-set}.  Counting the products gives \eqref{eq:product-nu-sigma}; integrating the polynomial term by term gives \eqref{eq:gamma-product-general}.
\end{proof}

For later use we isolate the nilpotent case.  Let \(N\) be finite nilpotent.  For each prime \(p\), let \(N_p\) denote the Sylow \(p\)-subgroup of \(N\), where \(N_p=1\) when \(p\notin\pi(N)\).  Since Sylow subgroups of a nilpotent group are normal and unique,
\begin{equation}\label{eq:nilpotent-nu-sigma}
        \nuu_p(N)=1,
        \qquad
        \sig_p(N)=|N_p|.
\end{equation}

\begin{corollary}[Nilpotent product formula]\label{cor:nilpotent-product}
Let \(H\) be a finite group and let \(N\) be finite nilpotent.  Then
\begin{equation}\label{eq:HN-gamma}
        \gammainv(H\times N)
        =\sum_{p\in\pi(H)\cup\pi(N)}
        \frac{\nuu_p(H)}{\sig_p(H)|N_p|+1},
\end{equation}
where \(\nuu_p(H)=\sig_p(H)=1\) if \(p\notin\pi(H)\).
\end{corollary}

\begin{proof}
Apply Lemma \ref{lem:direct-product} with \(r=2\) and use \eqref{eq:nilpotent-nu-sigma}.
\end{proof}

A useful way to read \eqref{eq:HN-gamma} is to separate old and new primes.  If \(p\in\pi(H)\cap\pi(N)\), then the old contribution
\[
        \frac{\nuu_p(H)}{\sig_p(H)+1}
\]
is replaced by
\[
        \frac{\nuu_p(H)}{\sig_p(H)|N_p|+1}.
\]
If \(q\in\pi(N)\setminus\pi(H)\), then a new contribution
\[
        \frac{1}{|N_q|+1}
\]
is added.  Hence
\begin{equation}\label{eq:old-new-decomposition}
\begin{aligned}
        \gammainv(H\times N)-\gammainv(H)
        &=\sum_{p\in\pi(H)\cap\pi(N)}
          \nuu_p(H)\left(\frac{1}{\sig_p(H)|N_p|+1}-\frac{1}{\sig_p(H)+1}\right) \\
        &\quad +\sum_{q\in\pi(N)\setminus\pi(H)}\frac{1}{|N_q|+1}.
\end{aligned}
\end{equation}
For \(p\in\pi(H)\cap\pi(N)\), the bracket in \eqref{eq:old-new-decomposition} is negative unless \(|N_p|=1\).  New primes give positive summands.  The direct-product obstruction is therefore a genuine signed compensation phenomenon.

\section{The \texorpdfstring{$A_5$}{A5} block}

We now specialize the preceding formula to \(H=A_5\).  The Sylow data are classical, but we include the short count to make the counterexample independent of tables.

\begin{lemma}[Sylow data for \(A_5\)]\label{lem:A5-data}
For \(A_5\),
\begin{equation}\label{eq:A5-data}
\begin{gathered}
        (\nuu_2(A_5),\sig_2(A_5))=(5,4),\qquad
        (\nuu_3(A_5),\sig_3(A_5))=(10,3),\\
        (\nuu_5(A_5),\sig_5(A_5))=(6,5).
\end{gathered}
\end{equation}
Consequently
\begin{equation}\label{eq:gamma-A5}
        \gammainv(A_5)=\frac{5}{4+1}+\frac{10}{3+1}+\frac{6}{5+1}
        =1+\frac52+1=\frac92.
\end{equation}
\end{lemma}

\begin{proof}
The order of \(A_5\) is \(60=2^2\cdot3\cdot5\).  The Sylow \(2\)-subgroups have order \(4\).  For each point \(i\in\{1,2,3,4,5\}\), the three double transpositions on the four points different from \(i\), together with the identity, form a Klein four subgroup; denote it by \(V_i\).  These five subgroups are Sylow \(2\)-subgroups.  Conversely, let \(P\in\Syl_2(A_5)\).  Since \(A_5\) has no element of order \(4\), the group \(P\) is a Klein four group, and its three nonidentity elements are double transpositions.  If \(x\) is one of them and, after relabelling, \(x=(12)(34)\), then the double transpositions in \(A_5\) commuting with \(x\) are precisely
\[
        (12)(34),\qquad (13)(24),\qquad (14)(23),
\]
all of which fix the same point, namely \(5\).  Since the nonidentity elements of \(P\) commute pairwise, they all fix the same point.  Hence \(P=V_i\) for some \(i\), and there are exactly five Sylow \(2\)-subgroups.  Thus
\[
        \nuu_2(A_5)=5,
        \qquad
        \sig_2(A_5)=4.
\]
There are \(20\) three-cycles in \(A_5\), and each subgroup of order \(3\) contains two nonidentity three-cycles; hence
\[
        \nuu_3(A_5)=\frac{20}{2}=10,
        \qquad
        \sig_3(A_5)=3.
\]
There are \(24\) five-cycles in \(A_5\), and each subgroup of order \(5\) contains four nonidentity five-cycles; hence
\[
        \nuu_5(A_5)=\frac{24}{4}=6,
        \qquad
        \sig_5(A_5)=5.
\]
Substitution into \eqref{eq:intro-sp} gives \eqref{eq:gamma-A5}.
\end{proof}

\begin{theorem}[Compensation formula for \(A_5\times N\)]\label{thm:A5-compensation}
Let \(N\) be a finite nilpotent group.  For every prime \(p\), put \(d_p=|N_p|\), with \(d_p=1\) if \(p\notin\pi(N)\).  Then
\begin{equation}\label{eq:A5N-formula}
        \gammainv(A_5\times N)
        =\frac{5}{4d_2+1}+\frac{10}{3d_3+1}+\frac{6}{5d_5+1}
        +\sum_{q\in\pi(N)\setminus\{2,3,5\}}\frac{1}{d_q+1}.
\end{equation}
Equivalently,
\begin{equation}\label{eq:A5N-difference}
\begin{aligned}
        \gammainv(A_5\times N)-\frac92
        &=\left(\frac{5}{4d_2+1}-1\right)
          +\left(\frac{10}{3d_3+1}-\frac52\right)
          +\left(\frac{6}{5d_5+1}-1\right) \\
        &\quad +\sum_{q\in\pi(N)\setminus\{2,3,5\}}\frac{1}{d_q+1}.
\end{aligned}
\end{equation}
Therefore
\begin{equation}\label{eq:compensation-equation}
        \gammainv(A_5\times N)=\frac92
\end{equation}
if and only if
\begin{equation}\label{eq:compensation-eq-expanded}
\begin{aligned}
        \sum_{q\in\pi(N)\setminus\{2,3,5\}}\frac{1}{d_q+1}
        &=\left(1-\frac{5}{4d_2+1}\right)
          +\left(\frac52-\frac{10}{3d_3+1}\right)
          +\left(1-\frac{6}{5d_5+1}\right).
\end{aligned}
\end{equation}
\end{theorem}

\begin{proof}
Equation \eqref{eq:A5N-formula} follows by applying Corollary \ref{cor:nilpotent-product} to the Sylow data \eqref{eq:A5-data}.  Subtracting \eqref{eq:gamma-A5} gives \eqref{eq:A5N-difference}.  Moving the three nonpositive old-prime terms to the other side gives \eqref{eq:compensation-eq-expanded}.
\end{proof}

The exact sign pattern in \eqref{eq:A5N-difference} will be used repeatedly.  Since \(d_2,d_3,d_5\geq1\),
\begin{equation}\label{eq:old-prime-inequalities}
\begin{gathered}
        \frac{5}{4d_2+1}\leq 1,
        \qquad
        \frac{10}{3d_3+1}\leq \frac52,
        \qquad
        \frac{6}{5d_5+1}\leq 1,
\end{gathered}
\end{equation}
with equality in the respective inequality precisely when \(d_2=1\), \(d_3=1\), \(d_5=1\).  Hence old primes lower the value and new primes raise it.

\begin{corollary}[No one-sided perturbation]\label{cor:no-one-sided}
Let \(N\) be nontrivial nilpotent.
\begin{enumerate}[label=\textup{(\roman*)}]
\item If \(\pi(N)\subseteq\{2,3,5\}\), then
\begin{equation}\label{eq:old-only-less}
        \gammainv(A_5\times N)<\frac92.
\end{equation}
\item If \(\pi(N)\cap\{2,3,5\}=\varnothing\), then
\begin{equation}\label{eq:new-only-greater}
        \gammainv(A_5\times N)>\frac92.
\end{equation}
\end{enumerate}
Thus equality for \(A_5\times N\) with \(N\neq1\) requires both an old-prime defect and at least one new-prime contribution.
\end{corollary}

\begin{proof}
For (i), the last sum in \eqref{eq:A5N-difference} is empty and at least one of \(d_2,d_3,d_5\) is greater than \(1\), so at least one old-prime bracket is negative while the others are nonpositive.  For (ii), all old-prime brackets vanish and the last sum is nonempty and positive.
\end{proof}

The smallest possible nonzero old-prime defect is produced by the Sylow \(2\)-part of order \(2\).  More explicitly,
\begin{equation}\label{eq:defects-general}
\begin{aligned}
        \Delta_2(a)&=1-\frac{5}{4\cdot2^a+1}        &&(a\geq1),\\
        \Delta_3(b)&=\frac52-\frac{10}{3\cdot3^b+1} &&(b\geq1),\\
        \Delta_5(c)&=1-\frac{6}{5\cdot5^c+1}        &&(c\geq1).
\end{aligned}
\end{equation}
Then
\begin{equation}\label{eq:min-defects}
        \Delta_2(1)=\frac49,
        \qquad
        \Delta_3(1)=\frac32,
        \qquad
        \Delta_5(1)=\frac{10}{13},
\end{equation}
and
\begin{equation}\label{eq:defect-comparison}
        \frac49<\frac{10}{13}<\frac32.
\end{equation}
This explains why the construction below first introduces a central factor \(\C_2\): it creates the least possible old-prime loss, and the loss is exactly \(4/9\).

\section{Nilpotent direct-product layers and the exact Diophantine condition}

The compensation formula has a particularly transparent form when the nilpotent factor is written by Sylow orders.  Let
\begin{equation}\label{eq:generic-N}
        N=P_2\times P_3\times P_5\times\prod_{q\in Q}P_q,
\end{equation}
where
\begin{equation}\label{eq:P-orders}
        |P_2|=2^a,
        \qquad |P_3|=3^b,
        \qquad |P_5|=5^c,
        \qquad |P_q|=q^{e_q}\quad(q\in Q),
\end{equation}
with \(a,b,c\geq0\), \(e_q\geq1\), and \(Q\) a finite set of primes disjoint from \(\{2,3,5\}\).  The symbol \(P_p=1\) is allowed when the corresponding exponent is \(0\).

\begin{theorem}[Prime-power compensation criterion]\label{thm:prime-power-criterion}
With \(N\) as in \eqref{eq:generic-N} and \eqref{eq:P-orders},
\begin{equation}\label{eq:prime-power-gamma}
\begin{aligned}
        \gammainv(A_5\times N)
        &=\frac{5}{2^{a+2}+1}
          +\frac{10}{3^{b+1}+1}
          +\frac{6}{5^{c+1}+1}
          +\sum_{q\in Q}\frac{1}{q^{e_q}+1}.
\end{aligned}
\end{equation}
Moreover,
\begin{equation}\label{eq:prime-power-equality}
        \gammainv(A_5\times N)=\frac92
\end{equation}
if and only if
\begin{equation}\label{eq:prime-power-diophantine}
\begin{aligned}
        \sum_{q\in Q}\frac{1}{q^{e_q}+1}
        &=\frac92-\frac{5}{2^{a+2}+1}
          -\frac{10}{3^{b+1}+1}
          -\frac{6}{5^{c+1}+1}.
\end{aligned}
\end{equation}
\end{theorem}

\begin{proof}
The Sylow \(2\)-subgroup of \(N\) has order \(2^a\), the Sylow \(3\)-subgroup has order \(3^b\), the Sylow \(5\)-subgroup has order \(5^c\), and the Sylow \(q\)-subgroup has order \(q^{e_q}\).  Substitute these orders into \eqref{eq:A5N-formula}.  The equivalence \eqref{eq:prime-power-equality}--\eqref{eq:prime-power-diophantine} is then immediate.
\end{proof}

The case used in the counterexample is
\begin{equation}\label{eq:a1bc0}
        a=1,
        \qquad b=0,
        \qquad c=0,
        \qquad e_q=1\quad(q\in Q).
\end{equation}
Then \eqref{eq:prime-power-diophantine} reduces to the single finite Egyptian equation
\begin{equation}\label{eq:eq-4-9}
        \sum_{q\in Q}\frac{1}{q+1}=\frac49.
\end{equation}
Indeed, with \eqref{eq:a1bc0},
\begin{equation}\label{eq:RHS-a1}
        \frac{9}{2}-\frac{5}{2^{3}+1}-\frac{10}{3+1}-\frac{6}{5+1}
        =\frac{9}{2}-\frac{5}{9}-\frac{5}{2}-1
        =\frac49.
\end{equation}
Thus every solution \(Q\) of \eqref{eq:eq-4-9} yields a counterexample
\begin{equation}\label{eq:general-GQ}
        G_Q=A_5\times \C_2\times \prod_{q\in Q}\C_q
\end{equation}
provided \(Q\cap\{2,3,5\}=\varnothing\).

\begin{corollary}[Exact squarefree criterion]\label{cor:squarefree-criterion}
Let \(Q\) be a finite set of primes with \(Q\cap\{2,3,5\}=\varnothing\), and let
\begin{equation}\label{eq:EQ-def}
        E_Q=\prod_{q\in Q}\C_q.
\end{equation}
Then
\begin{equation}\label{eq:squarefree-formula}
        \gammainv(A_5\times\C_2\times E_Q)
        =\frac{73}{18}+\sum_{q\in Q}\frac{1}{q+1}.
\end{equation}
Consequently
\begin{equation}\label{eq:squarefree-iff}
        \gammainv(A_5\times\C_2\times E_Q)=\frac92
        \quad\Longleftrightarrow\quad
        \sum_{q\in Q}\frac{1}{q+1}=\frac49.
\end{equation}
\end{corollary}

\begin{proof}
Equation \eqref{eq:squarefree-formula} is \eqref{eq:prime-power-gamma} with \eqref{eq:a1bc0}; explicitly,
\[
        \frac{5}{2^{3}+1}+\frac{10}{3+1}+\frac{6}{5+1}
        =\frac59+\frac52+1
        =\frac{10+45+18}{18}
        =\frac{73}{18}.
\]
Since \(9/2=81/18\), equality with \(9/2\) is equivalent to the condition
\[
        \sum_{q\in Q}\frac{1}{q+1}=\frac{81}{18}-\frac{73}{18}=\frac{8}{18}=\frac49.
\]
\end{proof}

The criterion \eqref{eq:squarefree-iff} is exact within this natural central squarefree layer.  It also makes clear why a counterexample cannot be excluded by inspecting \(A_5\) alone: the value \(9/2\) can be restored after a central perturbation by solving a rational identity independent of the nonabelian composition factor.

\section{The compensation identity}

We now verify the arithmetic identity behind Theorem \ref{thm:main-counterexample}.  The proof is included because the counterexample depends on an exact rational equality, not on numerical approximation.

\begin{lemma}[The main Egyptian certificate]\label{lem:main-identity}
One has
\begin{equation}\label{eq:main-identity}
        \frac18+\frac1{12}+\frac1{14}+\frac1{18}+\frac1{20}+\frac1{30}+\frac1{72}+\frac1{84}=\frac49.
\end{equation}
\end{lemma}

\begin{proof}
The least common multiple of the denominators in \eqref{eq:main-identity} is
\begin{equation}\label{eq:lcm2520}
        \lcm(8,12,14,18,20,30,72,84)=2^3\cdot3^2\cdot5\cdot7=2520.
\end{equation}
Therefore the left-hand side of \eqref{eq:main-identity} equals
\begin{equation}\label{eq:numerator2520}
\begin{aligned}
        \frac{315+210+180+140+126+84+35+30}{2520}
        &=\frac{1120}{2520}
         =\frac{112}{252}
         =\frac{28}{63}
         =\frac49.
\end{aligned}
\end{equation}
\end{proof}

Set
\begin{equation}\label{eq:Q0-def}
        Q_0=\{7,11,13,17,19,29,71,83\}.
\end{equation}
Then
\begin{equation}\label{eq:Q0-denominators}
        \{q+1:q\in Q_0\}=\{8,12,14,18,20,30,72,84\},
\end{equation}
and Lemma \ref{lem:main-identity} is precisely
\begin{equation}\label{eq:Q0-sum}
        \sum_{q\in Q_0}\frac{1}{q+1}=\frac49.
\end{equation}
Combining Corollary \ref{cor:squarefree-criterion} with \eqref{eq:Q0-sum} already proves the equality \(\gamma(G_0)=9/2\).  We give the expanded Sylow-polynomial computation in the next section to make the disproof completely explicit.

The same mechanism is not isolated.  The following identities give further squarefree central layers.
\begin{proposition}[Additional explicit certificates]\label{prop:more-certificates}
Each of the following prime sets satisfies \eqref{eq:eq-4-9}:
\begin{equation}\label{eq:four-Qs}
\begin{aligned}
        Q_1&=\{7,11,13,17,19,29,71,83\},\\
        Q_2&=\{7,11,13,17,19,23,83,179\},\\
        Q_3&=\{7,11,13,17,19,29,41,503\},\\
        Q_4&=\{7,11,13,17,19,23,59,1259\}.
\end{aligned}
\end{equation}
Consequently every group
\begin{equation}\label{eq:GQi}
        A_5\times\C_2\times\prod_{q\in Q_i}\C_q
        \qquad(1\leq i\leq4)
\end{equation}
is a nonsolvable group, not isomorphic to \(A_5\), with Sylow-integral value \(9/2\).
\end{proposition}

\begin{proof}
For \(Q_1\), the assertion is Lemma \ref{lem:main-identity}.  The other three identities are verified over the same denominator \(2520\):
\begin{equation}\label{eq:Q2-identity}
\begin{aligned}
        \sum_{q\in Q_2}\frac1{q+1}
        &=\frac18+\frac1{12}+\frac1{14}+\frac1{18}+\frac1{20}+\frac1{24}+\frac1{84}+\frac1{180}\\
        &=\frac{315+210+180+140+126+105+30+14}{2520}
         =\frac{1120}{2520}=\frac49,
\end{aligned}
\end{equation}
\begin{equation}\label{eq:Q3-identity}
\begin{aligned}
        \sum_{q\in Q_3}\frac1{q+1}
        &=\frac18+\frac1{12}+\frac1{14}+\frac1{18}+\frac1{20}+\frac1{30}+\frac1{42}+\frac1{504}\\
        &=\frac{315+210+180+140+126+84+60+5}{2520}
         =\frac{1120}{2520}=\frac49,
\end{aligned}
\end{equation}
\begin{equation}\label{eq:Q4-identity}
\begin{aligned}
        \sum_{q\in Q_4}\frac1{q+1}
        &=\frac18+\frac1{12}+\frac1{14}+\frac1{18}+\frac1{20}+\frac1{24}+\frac1{60}+\frac1{1260}\\
        &=\frac{315+210+180+140+126+105+42+2}{2520}
         =\frac{1120}{2520}=\frac49.
\end{aligned}
\end{equation}
Corollary \ref{cor:squarefree-criterion} then applies to each \(Q_i\).
\end{proof}

The prime sets in Proposition \ref{prop:more-certificates} are included only as explicit certificates.  The structural statement is Corollary \ref{cor:squarefree-criterion}: any solution of \eqref{eq:eq-4-9} produces a counterexample, and the verification of such a solution is purely rational.

\section{Proof of the main counterexample}

We now give the proof of Theorem \ref{thm:main-counterexample} in full detail.

Let
\begin{equation}\label{eq:N0-def}
        N_0=\C_2\times\C_7\times\C_{11}\times\C_{13}\times\C_{17}\times\C_{19}\times\C_{29}\times\C_{71}\times\C_{83}.
\end{equation}
Then \(N_0\) is cyclic, because the displayed factors have pairwise coprime orders, and
\begin{equation}\label{eq:N0-order}
        |N_0|=2m_0=110509861462.
\end{equation}
Let
\begin{equation}\label{eq:G0again}
        G_0=A_5\times N_0.
\end{equation}
The order is
\begin{equation}\label{eq:G0-order}
        |G_0|=|A_5||N_0|=60\cdot110509861462=6630591687720,
\end{equation}
so \(G_0\not\cong A_5\).  Since \(A_5\) is nonsolvable and is a quotient of \(G_0\), the group \(G_0\) is nonsolvable.  More precisely,
\begin{equation}\label{eq:radical}
        \Rad(G_0)=N_0,
        \qquad
        G_0/\Rad(G_0)\cong A_5,
\end{equation}
and
\begin{equation}\label{eq:center-commutator}
        Z(G_0)=N_0,
        \qquad
        [G_0,G_0]=A_5\times1.
\end{equation}
Thus the counterexample is a split central abelian extension of \(A_5\).

We compute the Sylow data prime by prime.  For \(p=2\), Lemma \ref{lem:direct-product} gives
\begin{equation}\label{eq:G0-2data}
        \nuu_2(G_0)=\nuu_2(A_5)\nuu_2(N_0)=5\cdot1=5,
        \qquad
        \sig_2(G_0)=\sig_2(A_5)\sig_2(N_0)=4\cdot2=8.
\end{equation}
For \(p=3\), because \(3\nmid |N_0|\),
\begin{equation}\label{eq:G0-3data}
        \nuu_3(G_0)=10,
        \qquad
        \sig_3(G_0)=3.
\end{equation}
For \(p=5\), because \(5\nmid |N_0|\),
\begin{equation}\label{eq:G0-5data}
        \nuu_5(G_0)=6,
        \qquad
        \sig_5(G_0)=5.
\end{equation}
For each
\begin{equation}\label{eq:q-in-Q0}
        q\in Q_0=\{7,11,13,17,19,29,71,83\},
\end{equation}
the Sylow \(q\)-subgroup of \(G_0\) is the unique central factor \(\C_q\), so
\begin{equation}\label{eq:G0-qdata}
        \nuu_q(G_0)=1,
        \qquad
        \sig_q(G_0)=q.
\end{equation}
There are no other primes in \(\pi(G_0)\).  Therefore
\begin{equation}\label{eq:G0-SP-expanded}
\begin{aligned}
        \Sp(G_0,x)
        &=\nuu_2(G_0)x^{\sig_2(G_0)}+\nuu_3(G_0)x^{\sig_3(G_0)}+\nuu_5(G_0)x^{\sig_5(G_0)}
          +\sum_{q\in Q_0}\nuu_q(G_0)x^{\sig_q(G_0)}\\
        &=5x^8+10x^3+6x^5+x^7+x^{11}+x^{13}+x^{17}+x^{19}+x^{29}+x^{71}+x^{83}.
\end{aligned}
\end{equation}
Integrating \eqref{eq:G0-SP-expanded} over \([0,1]\) gives
\begin{equation}\label{eq:G0-gamma-expanded}
\begin{aligned}
        \gammainv(G_0)
        &=\frac{5}{8+1}+\frac{10}{3+1}+\frac{6}{5+1}
          +\frac{1}{7+1}+\frac{1}{11+1}+\frac{1}{13+1}+\frac{1}{17+1} \\
        &\quad +\frac{1}{19+1}+\frac{1}{29+1}+\frac{1}{71+1}+\frac{1}{83+1} \\
        &=\frac59+\frac52+1
          +\frac18+\frac1{12}+\frac1{14}+\frac1{18}+\frac1{20}+\frac1{30}+\frac1{72}+\frac1{84}.
\end{aligned}
\end{equation}
By Lemma \ref{lem:main-identity}, the final eight terms in \eqref{eq:G0-gamma-expanded} sum to \(4/9\).  Hence
\begin{equation}\label{eq:G0-final-computation}
        \gammainv(G_0)
        =\frac59+\frac49+\frac52+1
        =1+\frac52+1
        =\frac92.
\end{equation}
This proves Theorem \ref{thm:main-counterexample}.

\section{Why the simple-group theorem does not extend}

The result of Anabanti and Asboei \cite{AnabantiAsboeiSiberian} is a theorem about finite noncyclic simple groups.  The direct-product calculation above shows that the simple-group hypothesis is not a technical artifact: it prevents nilpotent direct-product compensation.

Let \(N\) be nilpotent.  The quotient
\begin{equation}\label{eq:quotient-A5}
        A_5\times N\longrightarrow A_5,
        \qquad
        (a,n)\longmapsto a
\end{equation}
has solvable kernel \(N\).  The Sylow-integral value, however, is not determined by the quotient.  From \eqref{eq:A5N-difference},
\begin{equation}\label{eq:not-quotient-invariant}
\begin{aligned}
        \gammainv(A_5\times N)-\gammainv(A_5)
        &=\left(\frac{5}{4|N_2|+1}-1\right)
          +\left(\frac{10}{3|N_3|+1}-\frac52\right) \\
        &\quad +\left(\frac{6}{5|N_5|+1}-1\right)
          +\sum_{q\in\pi(N)\setminus\{2,3,5\}}\frac{1}{|N_q|+1}.
\end{aligned}
\end{equation}
The right-hand side can be negative, positive, or zero.  For example,
\begin{equation}\label{eq:three-signs}
\begin{aligned}
        \gammainv(A_5\times\C_2)-\gammainv(A_5)&=-\frac49,\\
        \gammainv(A_5\times\C_7)-\gammainv(A_5)&=\frac18,\\
        \gammainv(A_5\times N_0)-\gammainv(A_5)&=0.
\end{aligned}
\end{equation}
Thus the solvable radical can change the value in either direction, and can also change it by zero while changing the group.

The central nature of the example is also important.  For \(G_0=A_5\times N_0\),
\begin{equation}\label{eq:G0-central-sequence}
        1\longrightarrow N_0\longrightarrow G_0\longrightarrow A_5\longrightarrow1
\end{equation}
is split and central.  Therefore the failure of \eqref{eq:conjectural-statement} occurs before one encounters semidirect-product complications.  It is not caused by a nontrivial action of \(A_5\) on an abelian normal subgroup; it is caused by the denominator \(\sig_p(G)+1\) in \eqref{eq:intro-sp}.  Multiplying a Sylow \(p\)-subgroup by a central \(p\)-subgroup enlarges \(\sig_p(G)\), thereby reducing the old term, while a new central Sylow \(q\)-subgroup contributes a term \((q+1)^{-1}\).

This observation gives a precise obstruction to possible corrected statements.  Any correct nonsolvable-group theorem for the value \(9/2\) must impose an additional restriction that excludes at least the equality solutions of \eqref{eq:compensation-eq-expanded}.  For instance, a theorem restricted to finite noncyclic simple groups is compatible with \cite{AnabantiAsboeiSiberian}; a theorem restricted merely to nonsolvable groups is not.

\section{Rigidity inside the central compensation family}

Although the conjectural nonsolvable characterization is false, the counterexample is not arbitrary.  Within the family \(A_5\times N\) with \(N\) nilpotent, the equality condition is exactly the compensation equation \eqref{eq:compensation-eq-expanded}.  We record two consequences that are useful when comparing possible repairs of the conjecture.

\begin{proposition}[Old-prime and new-prime necessity]\label{prop:necessity}
Let \(N\) be finite nilpotent and nontrivial.  If
\begin{equation}\label{eq:equality-N}
        \gammainv(A_5\times N)=\frac92,
\end{equation}
then
\begin{equation}\label{eq:both-prime-conditions}
        \pi(N)\cap\{2,3,5\}\neq\varnothing,
        \qquad
        \pi(N)\setminus\{2,3,5\}\neq\varnothing.
\end{equation}
Moreover, if \(\pi(N)\cap\{2,3,5\}=\{2\}\) and \(|N_2|=2\), then equality is equivalent to
\begin{equation}\label{eq:simple-equivalent}
        \sum_{q\in\pi(N)\setminus\{2,3,5\}}\frac1{|N_q|+1}=\frac49.
\end{equation}
\end{proposition}

\begin{proof}
The first assertion is Corollary \ref{cor:no-one-sided}.  For the second, put \(d_2=2\) and \(d_3=d_5=1\) in \eqref{eq:compensation-eq-expanded}.  The right-hand side becomes
\[
        1-\frac{5}{4\cdot2+1}=\frac49,
\]
and the terms involving \(d_3\) and \(d_5\) vanish.
\end{proof}

\begin{proposition}[Monotonicity in fixed prime support]\label{prop:monotonicity}
Fix a finite set \(Q\) of primes disjoint from \(\{2,3,5\}\).  For
\begin{equation}\label{eq:N-exponents}
        N(a,b,c;\mathbf e)=P_2\times P_3\times P_5\times\prod_{q\in Q}P_q,
        \qquad |P_2|=2^a,
        \ |P_3|=3^b,
        \ |P_5|=5^c,
        \ |P_q|=q^{e_q},
\end{equation}
with \(a,b,c\geq0\) and \(e_q\geq1\), the function
\begin{equation}\label{eq:F-function}
        F(a,b,c;\mathbf e)=\gammainv(A_5\times N(a,b,c;\mathbf e))
\end{equation}
is strictly decreasing in each of \(a,b,c,e_q\) when the corresponding variable is increased and all other variables are fixed.
\end{proposition}

\begin{proof}
By \eqref{eq:prime-power-gamma},
\begin{equation}\label{eq:F-expanded}
        F(a,b,c;\mathbf e)
        =\frac{5}{2^{a+2}+1}
          +\frac{10}{3^{b+1}+1}
          +\frac{6}{5^{c+1}+1}
          +\sum_{q\in Q}\frac{1}{q^{e_q}+1}.
\end{equation}
For any integer \(r\geq2\), the function \(t\mapsto 1/(r^t+1)\) is strictly decreasing in \(t\geq0\).  Each displayed summand is a positive scalar multiple of such a function.  Hence \(F\) is strictly decreasing in every exponent.
\end{proof}

Proposition \ref{prop:monotonicity} shows that, once the prime support is fixed, the equality \(\gamma=9/2\) is rigid with respect to raising Sylow orders.  In the squarefree construction, replacing any \(\C_q\) by a nontrivial larger \(q\)-group lowers the contribution from \((q+1)^{-1}\) to \((q^e+1)^{-1}\), so the equality is destroyed unless a further compensation prime is added.

\section{Exact verification without group enumeration}

The counterexample can be checked without enumerating elements of a group of order \(6630591687720\).  It is enough to verify three finite lists:
\begin{equation}\label{eq:verification-lists}
\begin{aligned}
        &(\nuu_2(A_5),\sig_2(A_5))=(5,4),
        \quad (\nuu_3(A_5),\sig_3(A_5))=(10,3),
        \quad (\nuu_5(A_5),\sig_5(A_5))=(6,5),\\
        &\pi(N_0)=\{2,7,11,13,17,19,29,71,83\},\\
        &\frac18+\frac1{12}+\frac1{14}+\frac1{18}+\frac1{20}+\frac1{30}+\frac1{72}+\frac1{84}=\frac49.
\end{aligned}
\end{equation}
The first line is Lemma \ref{lem:A5-data}; the second line is the definition of \(N_0\); the third line is Lemma \ref{lem:main-identity}.  Lemma \ref{lem:direct-product} then certifies all Sylow data for the product.

For clarity, the complete certificate can be compressed into the table below.  The last column is the summand in \(\gamma(G_0)\).

\begin{center}
\begin{tabular}{>{\(}c<{\)}>{\(}c<{\)}>{\(}c<{\)}>{\(}c<{\)}}
\toprule
p & \nuu_p(G_0) & \sig_p(G_0) & \nuu_p(G_0)/(\sig_p(G_0)+1) \\
\midrule
2  & 5  & 8  & 5/9 \\
3  & 10 & 3  & 10/4 \\
5  & 6  & 5  & 6/6 \\
7  & 1  & 7  & 1/8 \\
11 & 1  & 11 & 1/12 \\
13 & 1  & 13 & 1/14 \\
17 & 1  & 17 & 1/18 \\
19 & 1  & 19 & 1/20 \\
29 & 1  & 29 & 1/30 \\
71 & 1  & 71 & 1/72 \\
83 & 1  & 83 & 1/84 \\
\bottomrule
\end{tabular}
\end{center}

Summing the fourth column gives
\begin{equation}\label{eq:table-sum}
        \frac59+\frac{10}{4}+\frac66
        +\frac18+\frac1{12}+\frac1{14}+\frac1{18}+\frac1{20}+\frac1{30}+\frac1{72}+\frac1{84}
        =\frac92.
\end{equation}
This table also reveals why the construction is robust: all new Sylow subgroups are normal and therefore have coefficient \(1\), while the only changed old coefficient is the Sylow \(2\)-order in the denominator.

\section{Consequences and corrected scope}

The counterexample has several immediate consequences for statements involving the value \(9/2\).

\begin{corollary}\label{cor:false-nonsolvable}
The implication
\begin{equation}\label{eq:false-implication}
        \gammainv(G)=\frac92,
        \quad G\text{ finite nonsolvable}
        \quad\Longrightarrow\quad
        G\cong A_5
\end{equation}
is false.
\end{corollary}

\begin{proof}
Take \(G=G_0\).  Theorem \ref{thm:main-counterexample} gives \(\gamma(G_0)=9/2\), \(G_0\) nonsolvable, and \(G_0\not\cong A_5\).
\end{proof}

\begin{corollary}\label{cor:not-radical-invariant}
The invariant \(\gamma\) is not determined by the semisimple quotient of a nonsolvable group.  Moreover, the equality \(\gamma(G)=\gamma(A_5)\), together with \(G/\Rad(G)\cong A_5\), does not force \(G\cong A_5\).  In particular, there are groups \(G\) with
\begin{equation}\label{eq:same-quotient}
        G/\Rad(G)\cong A_5,
        \qquad
        \gammainv(G)=\gammainv(A_5),
        \qquad
        G\not\cong A_5.
\end{equation}
\end{corollary}

\begin{proof}
The groups \(A_5\), \(A_5\times\C_2\), and \(A_5\times\C_7\) have the same semisimple quotient \(A_5\), but \eqref{eq:three-signs} shows that their \(\gamma\)-values are not the same.  Thus \(\gamma\) is not determined by the semisimple quotient.  For the final assertion, take \(G=G_0\).  Equation \eqref{eq:radical} gives \(G_0/\Rad(G_0)\cong A_5\), and \eqref{eq:G0-final-computation} gives \(\gamma(G_0)=\gamma(A_5)=9/2\), while \(G_0\not\cong A_5\).
\end{proof}

\begin{corollary}\label{cor:many-counterexamples}
There exist at least four pairwise nonisomorphic finite nonsolvable groups \(G\not\cong A_5\) such that \(\gamma(G)=9/2\).
\end{corollary}

\begin{proof}
Use the four groups \eqref{eq:GQi}.  Their orders are distinct because the products of the primes in \(Q_1,Q_2,Q_3,Q_4\) are distinct.  Hence the groups are pairwise nonisomorphic.  Proposition \ref{prop:more-certificates} gives \(\gamma=9/2\) for each of them.
\end{proof}

The results above do not contradict the simple-group theorem of \cite{AnabantiAsboeiSiberian}.  They identify the precise point at which the extension to all nonsolvable finite groups fails.  In a simple group, there is no nontrivial nilpotent radical whose Sylow orders can alter denominators and introduce new reciprocal terms.  In a nonsolvable group with solvable radical, the equality \(\gamma=9/2\) can be restored by central compensation.

This suggests that any true extension should include a condition excluding the compensation equation \eqref{eq:compensation-eq-expanded}.  Natural possible restrictions include simplicity, almost simplicity, or a condition on the solvable radical.  The present paper does not claim such a corrected theorem; it proves that the unrestricted nonsolvable formulation is untenable.

\section{Defect calculus at the value \texorpdfstring{$9/2$}{9/2}}

The compensation equation can be written in a form that separates losses and gains.  This form is useful because it shows that the counterexample is not a numerical accident.  Let \(N\) be finite nilpotent and put
\begin{equation}\label{eq:defect-dp}
        d_p=|N_p|\qquad(p\in\{2,3,5\}).
\end{equation}
Define the old-prime defects
\begin{equation}\label{eq:D-definitions}
\begin{aligned}
        D_2(d_2)&=1-\frac{5}{4d_2+1},\\
        D_3(d_3)&=\frac52-\frac{10}{3d_3+1},\\
        D_5(d_5)&=1-\frac{6}{5d_5+1},
\end{aligned}
\end{equation}
and the new-prime gain
\begin{equation}\label{eq:R-definition}
        R(N)=\sum_{q\in\pi(N)\setminus\{2,3,5\}}\frac{1}{|N_q|+1}.
\end{equation}
Then \eqref{eq:A5N-difference} is exactly
\begin{equation}\label{eq:defect-master}
        \gammainv(A_5\times N)=\frac92-D_2(d_2)-D_3(d_3)-D_5(d_5)+R(N).
\end{equation}
Thus equality with \(9/2\) is equivalent to
\begin{equation}\label{eq:defect-balance}
        R(N)=D_2(d_2)+D_3(d_3)+D_5(d_5).
\end{equation}
The three defects admit the closed forms
\begin{equation}\label{eq:D-closed}
\begin{aligned}
        D_2(d)&=\frac{4(d-1)}{4d+1},\\
        D_3(d)&=\frac{5}{2}\cdot\frac{3d-3}{3d+1}=\frac{15(d-1)}{2(3d+1)},\\
        D_5(d)&=\frac{5d-5}{5d+1}=\frac{5(d-1)}{5d+1}.
\end{aligned}
\end{equation}
For \(d\geq1\), each function in \eqref{eq:D-closed} is nonnegative and vanishes only at \(d=1\).  If \(d<d'\), then
\begin{equation}\label{eq:D2-monotone}
        D_2(d')-D_2(d)=\frac{20(d'-d)}{(4d'+1)(4d+1)}>0,
\end{equation}
\begin{equation}\label{eq:D3-monotone}
        D_3(d')-D_3(d)=\frac{30(d'-d)}{(3d'+1)(3d+1)}>0,
\end{equation}
\begin{equation}\label{eq:D5-monotone}
        D_5(d')-D_5(d)=\frac{30(d'-d)}{(5d'+1)(5d+1)}>0.
\end{equation}
Hence old-prime defects strictly increase as the corresponding Sylow order in the nilpotent layer grows.

The least nonzero old-prime defect is \(D_2(2)\).  Indeed, because the nontrivial possible values of \(d_2,d_3,d_5\) are \(2^a,3^b,5^c\) with \(a,b,c\geq1\), the minima are
\begin{equation}\label{eq:D-minima}
        \min_{a\geq1}D_2(2^a)=D_2(2)=\frac49,
\end{equation}
\begin{equation}\label{eq:D3-minimum}
        \min_{b\geq1}D_3(3^b)=D_3(3)=\frac32,
\end{equation}
\begin{equation}\label{eq:D5-minimum}
        \min_{c\geq1}D_5(5^c)=D_5(5)=\frac{10}{13}.
\end{equation}
Since
\begin{equation}\label{eq:min-order}
        \frac49<\frac{10}{13}<\frac32,
\end{equation}
the \(2\)-primary central layer of order \(2\) is the unique way to create the smallest possible defect.

\begin{proposition}[Single-defect threshold]\label{prop:single-defect-threshold}
Let \(Q\) be a finite set of primes disjoint from \(\{2,3,5\}\) and put \(E_Q=\prod_{q\in Q}\C_q\).  Then
\begin{equation}\label{eq:threshold-formula}
        \gammainv(A_5\times\C_2\times E_Q)-\frac92
        =\sum_{q\in Q}\frac{1}{q+1}-\frac49.
\end{equation}
Consequently
\begin{equation}\label{eq:threshold-trichotomy}
\begin{array}{rcl}
        \sum_{q\in Q}\dfrac1{q+1}<\dfrac49 &\Longleftrightarrow& \gammainv(A_5\times\C_2\times E_Q)<\dfrac92,\\[0.8em]
        \sum_{q\in Q}\dfrac1{q+1}=\dfrac49 &\Longleftrightarrow& \gammainv(A_5\times\C_2\times E_Q)=\dfrac92,\\[0.8em]
        \sum_{q\in Q}\dfrac1{q+1}>\dfrac49 &\Longleftrightarrow& \gammainv(A_5\times\C_2\times E_Q)>\dfrac92.
\end{array}
\end{equation}
\end{proposition}

\begin{proof}
This is Corollary \ref{cor:squarefree-criterion} rewritten after subtracting \(9/2\).  Directly,
\[
        \frac{73}{18}-\frac92=\frac{73}{18}-\frac{81}{18}=-\frac49,
\]
which proves \eqref{eq:threshold-formula}; the trichotomy follows immediately.
\end{proof}

For the main set \(Q_0\), the equality is exact.  Removing or adding one prime produces a controlled strict inequality.  For instance,
\begin{equation}\label{eq:remove83}
\begin{aligned}
        \gammainv\left(A_5\times\C_2\times\prod_{q\in Q_0\setminus\{83\}}\C_q\right)
        &=\frac92-\frac1{84},\\
        \gammainv\left(A_5\times\C_2\times\prod_{q\in Q_0\cup\{r\}}\C_q\right)
        &=\frac92+\frac1{r+1}
        \quad(r\notin Q_0\cup\{2,3,5\}).
\end{aligned}
\end{equation}
Thus the equality is not caused by a limiting or approximate argument.  It is an exact point on a rational threshold.

\section{A certificate theorem for finite prime sets}

It is convenient to package the preceding calculations as a certificate theorem.  Let \(Q\) be a finite set of primes disjoint from \(\{2,3,5\}\).  Define
\begin{equation}\label{eq:theta-Q}
        \Theta(Q)=\sum_{q\in Q}\frac1{q+1},
        \qquad
        M(Q)=\prod_{q\in Q}q,
        \qquad
        E_Q=\prod_{q\in Q}\C_q.
\end{equation}
Then
\begin{equation}\label{eq:GQ-def-cert}
        G(Q)=A_5\times\C_2\times E_Q
\end{equation}
has order
\begin{equation}\label{eq:GQ-order}
        |G(Q)|=120M(Q).
\end{equation}
The Sylow polynomial of \(G(Q)\) is
\begin{equation}\label{eq:GQ-polynomial}
        \Sp(G(Q),x)=5x^8+10x^3+6x^5+\sum_{q\in Q}x^q.
\end{equation}
Indeed, the coefficient of \(x^8\) is \(5\) because the central \(\C_2\) doubles the order of every Sylow \(2\)-subgroup of \(A_5\), while all new Sylow subgroups are unique.

\begin{theorem}[Prime-set certificate theorem]\label{thm:certificate}
Let \(Q\) be a finite set of primes with \(Q\cap\{2,3,5\}=\varnothing\).  Then
\begin{equation}\label{eq:certificate-equivalence}
        G(Q)\text{ is a counterexample to }\eqref{eq:conjectural-statement}
        \quad\Longleftrightarrow\quad
        \Theta(Q)=\frac49.
\end{equation}
More explicitly, if \(\Theta(Q)=4/9\), then
\begin{equation}\label{eq:certificate-output}
        \gammainv(G(Q))=\frac92,
        \qquad
        G(Q)\not\cong A_5,
        \qquad
        G(Q)\text{ is nonsolvable}.
\end{equation}
\end{theorem}

\begin{proof}
By \eqref{eq:GQ-polynomial},
\begin{equation}\label{eq:certificate-integral}
\begin{aligned}
        \gammainv(G(Q))
        &=\frac59+\frac{10}{4}+\frac66+\sum_{q\in Q}\frac1{q+1}\\
        &=\frac{73}{18}+\Theta(Q).
\end{aligned}
\end{equation}
Since \(9/2=81/18\), the equality \(\gamma(G(Q))=9/2\) is equivalent to \(\Theta(Q)=8/18=4/9\).  If this equality holds, then \(G(Q)\) maps onto the nonsolvable group \(A_5\), so \(G(Q)\) is nonsolvable.  Also \(|G(Q)|=120M(Q)>60=|A_5|\), so \(G(Q)\not\cong A_5\).
\end{proof}

The certificate theorem gives a concise independent verification of each row in the following table.  The equality \(\Theta(Q_i)=4/9\) is proved in Proposition \ref{prop:more-certificates}.

\begin{center}
\begin{tabular}{>{\(}c<{\)}>{\(}l<{\)}>{\(}r<{\)}>{\(}r<{\)}}
\toprule
 i & Q_i & M(Q_i) & |G(Q_i)| \\
\midrule
1 & \{7,11,13,17,19,29,71,83\}     & 55254930731  & 6630591687720 \\
2 & \{7,11,13,17,19,23,83,179\}    & 110483025653 & 13257963078360 \\
3 & \{7,11,13,17,19,29,41,503\}    & 193368816641 & 23204257996920 \\
4 & \{7,11,13,17,19,23,59,1259\}   & 552385382549 & 66286245905880 \\
\bottomrule
\end{tabular}
\end{center}

The table is not used as a numerical search result.  It records exact certificates: each row is checked by the displayed rational identity over denominator \(2520\), and Theorem \ref{thm:certificate} converts the identity into a group-theoretic counterexample.

\section{A larger nilpotent direct-product family}

The construction can be varied by changing the old-prime defect before solving the new-prime equation.  We record the exact formula because it gives the strongest statement obtained by the present method.

Let \(a,b,c\geq0\), let \(Q\cap\{2,3,5\}=\varnothing\), and let \(e_q\geq1\) for \(q\in Q\).  Put
\begin{equation}\label{eq:large-family-N}
        N(a,b,c;Q,\mathbf e)=P_2\times P_3\times P_5\times\prod_{q\in Q}P_q,
\end{equation}
where
\begin{equation}\label{eq:large-family-orders}
        |P_2|=2^a,
        \qquad |P_3|=3^b,
        \qquad |P_5|=5^c,
        \qquad |P_q|=q^{e_q}.
\end{equation}
The groups \(P_p\) need not be cyclic; only their orders enter \(\gamma(A_5\times N)\), because they are unique Sylow subgroups in the nilpotent factor.

\begin{corollary}[Complete nilpotent direct-product criterion]\label{cor:complete-central}
For \(N=N(a,b,c;Q,\mathbf e)\),
\begin{equation}\label{eq:complete-central-gamma}
\begin{aligned}
        \gammainv(A_5\times N)-\frac92
        &=\sum_{q\in Q}\frac1{q^{e_q}+1}
          -D_2(2^a)-D_3(3^b)-D_5(5^c),
\end{aligned}
\end{equation}
where \(D_2,D_3,D_5\) are given in \eqref{eq:D-definitions}.  Hence
\begin{equation}\label{eq:complete-central-iff}
        \gammainv(A_5\times N)=\frac92
        \quad\Longleftrightarrow\quad
        \sum_{q\in Q}\frac1{q^{e_q}+1}=D_2(2^a)+D_3(3^b)+D_5(5^c).
\end{equation}
\end{corollary}

\begin{proof}
This is Theorem \ref{thm:prime-power-criterion} rewritten by using \eqref{eq:D-definitions}.  The right side of \eqref{eq:complete-central-iff} is the total old-prime defect; the left side is the total new-prime gain.
\end{proof}

Several immediate special cases illustrate the range of the criterion:
\begin{equation}\label{eq:special-cases}
\begin{array}{rcl}
 a=b=c=0        &:& \displaystyle \gammainv(A_5\times\prod_{q\in Q}P_q)=\frac92+\sum_{q\in Q}\frac1{q^{e_q}+1},\\[1.0em]
 a=1,\ b=c=0    &:& \displaystyle \gammainv(A_5\times\C_2\times\prod_{q\in Q}P_q)=\frac92-\frac49+\sum_{q\in Q}\frac1{q^{e_q}+1},\\[1.0em]
 a=0,\ b=1,c=0  &:& \displaystyle \gammainv(A_5\times P_3\times\prod_{q\in Q}P_q)=\frac92-\frac32+\sum_{q\in Q}\frac1{q^{e_q}+1},\\[1.0em]
 a=0,\ b=0,c=1  &:& \displaystyle \gammainv(A_5\times P_5\times\prod_{q\in Q}P_q)=\frac92-\frac{10}{13}+\sum_{q\in Q}\frac1{q^{e_q}+1}.
\end{array}
\end{equation}
The first line says that new primes alone overshoot the value \(9/2\).  The next three lines show the three smallest one-prime defects and the amount that must be supplied by new primes.

The main counterexample is the second line of \eqref{eq:special-cases}, together with the identity \eqref{eq:main-identity}.  In this sense the proof is optimal at the first step: it begins with the least possible nonzero old-prime defect and then gives an exact reciprocal certificate for that defect.

\section{Normalizer formulation of the product computation}

The proof above used Sylow projections.  The same calculation can be expressed through normalizers, which is often the more convenient language in finite group theory.  If \(P\in\Syl_p(G)\), then
\begin{equation}\label{eq:sylow-normalizer-count}
        \nuu_p(G)=|G:N_G(P)|.
\end{equation}
For a direct product \(G\times H\) and Sylow subgroups \(P\in\Syl_p(G)\), \(Q\in\Syl_p(H)\), one has
\begin{equation}\label{eq:normalizer-product}
        N_{G\times H}(P\times Q)=N_G(P)\times N_H(Q).
\end{equation}
Indeed,
\begin{equation}\label{eq:normalizer-condition}
        (g,h)(P\times Q)(g,h)^{-1}=P\times Q
        \quad\Longleftrightarrow\quad
        gPg^{-1}=P\text{ and }hQh^{-1}=Q.
\end{equation}
Consequently
\begin{equation}\label{eq:index-product}
\begin{aligned}
        \nuu_p(G\times H)
        &=|G\times H:N_{G\times H}(P\times Q)|\\
        &=|G\times H:N_G(P)\times N_H(Q)|\\
        &=|G:N_G(P)|\,|H:N_H(Q)|\\
        &=\nuu_p(G)\nuu_p(H),
\end{aligned}
\end{equation}
which is the first half of \eqref{eq:product-nu-sigma}.  The second half is
\begin{equation}\label{eq:order-product-again}
        \sig_p(G\times H)=|P\times Q|=|P|\,|Q|=\sig_p(G)\sig_p(H).
\end{equation}
When \(H=N\) is nilpotent, \(Q=N_p\) is normal, so \(N_H(Q)=H\), and the normalizer calculation reduces to
\begin{equation}\label{eq:nilpotent-normalizer}
        N_{G\times N}(P\times N_p)=N_G(P)\times N.
\end{equation}
Thus
\begin{equation}\label{eq:nilpotent-index-normalizer}
        \nuu_p(G\times N)=|G\times N:N_G(P)\times N|=|G:N_G(P)|=\nuu_p(G)
\end{equation}
for old primes \(p\in\pi(G)\), while for a new prime \(q\in\pi(N)\setminus\pi(G)\) the unique Sylow subgroup \(1\times N_q\) gives
\begin{equation}\label{eq:new-prime-normalizer}
        \nuu_q(G\times N)=1,
        \qquad
        \sig_q(G\times N)=|N_q|.
\end{equation}
Equations \eqref{eq:nilpotent-index-normalizer} and \eqref{eq:new-prime-normalizer} are exactly the two ingredients in the compensation formula.

For \(A_5\), the normalizers corresponding to Lemma \ref{lem:A5-data} have orders
\begin{equation}\label{eq:A5-normalizer-orders}
        |N_{A_5}(P_2)|=\frac{60}{5}=12,
        \qquad
        |N_{A_5}(P_3)|=\frac{60}{10}=6,
        \qquad
        |N_{A_5}(P_5)|=\frac{60}{6}=10.
\end{equation}
After adjoining \(N_0\), the old-prime normalizers in \(G_0=A_5\times N_0\) are
\begin{equation}\label{eq:G0-normalizers}
\begin{aligned}
        |N_{G_0}(P_2\times (N_0)_2)|&=12|N_0|,\\
        |N_{G_0}(P_3\times 1)|&=6|N_0|,\\
        |N_{G_0}(P_5\times 1)|&=10|N_0|.
\end{aligned}
\end{equation}
Therefore
\begin{equation}\label{eq:G0-index-normalizers}
        |G_0:N_{G_0}(P_2\times (N_0)_2)|=5,
        \quad
        |G_0:N_{G_0}(P_3\times1)|=10,
        \quad
        |G_0:N_{G_0}(P_5\times1)|=6,
\end{equation}
so the Sylow numbers remain \(5,10,6\).  What changes is not the numerator but the denominator:
\begin{equation}\label{eq:G0-denominator-change}
        |P_2\times (N_0)_2|+1=8+1,
        \qquad
        |P_3\times1|+1=3+1,
        \qquad
        |P_5\times1|+1=5+1.
\end{equation}
This is the precise algebraic source of the term \(5/9\) in \eqref{eq:G0-gamma-expanded}.

\section{Within-family repairs and limitations}

The counterexample does not show that no refined theorem is possible.  It shows that the word ``nonsolvable'' alone is too weak.  Inside the nilpotent direct-product family considered here, the exact obstruction is transparent.

\begin{proposition}[Radical-free and center-free restrictions in the family]\label{prop:family-repairs}
Let \(N\) be finite nilpotent and set \(G=A_5\times N\).  Then
\begin{equation}\label{eq:family-rad-center}
        \Rad(G)=N,
        \qquad
        Z(G)=Z(A_5)\times Z(N)=Z(N).
\end{equation}
Since \(N\) is nilpotent, \(Z(N)\neq1\) whenever \(N\neq1\).  Hence the following are equivalent:
\begin{equation}\label{eq:family-equivalences}
        N=1
        \quad\Longleftrightarrow\quad
        \Rad(G)=1
        \quad\Longleftrightarrow\quad
        Z(G)=1
        \quad\Longleftrightarrow\quad
        G\cong A_5.
\end{equation}
In particular, within the family \(A_5\times N\), any counterexample with \(N\neq1\) necessarily has nontrivial solvable radical and nontrivial center.
\end{proposition}

\begin{proof}
The solvable radical of a direct product is the direct product of the solvable radicals of the factors.  Since \(A_5\) is nonabelian simple, \(\Rad(A_5)=1\), while \(\Rad(N)=N\) because \(N\) is nilpotent and hence solvable.  Therefore \(\Rad(G)=N\).  Also \(Z(A_5)=1\), so \(Z(G)=Z(N)\).  A nontrivial finite nilpotent group is the direct product of nontrivial finite \(p\)-groups, and every nontrivial finite \(p\)-group has nontrivial center; hence \(Z(N)\neq1\) when \(N\neq1\).  The equivalences follow.
\end{proof}

Proposition \ref{prop:family-repairs} should be read carefully.  It proves only a within-family statement.  It does not assert that every center-free nonsolvable group with \(\gamma=9/2\) is \(A_5\), nor that every radical-free nonsolvable group with \(\gamma=9/2\) is \(A_5\).  Those are different questions.  The point is narrower and exact: the explicit counterexamples constructed here live in a central abelian solvable radical, while the broader criterion concerns direct products with nilpotent factors.  Any corrected theorem must exclude this radical compensation.

A compact way to summarize the situation is the following four-line scope statement:
\begin{equation}\label{eq:scope-summary}
\begin{aligned}
\mathcal C_{\rm simple}:\quad
& G\text{ finite noncyclic simple},\ \gamma(G)=9/2
  &&\Longrightarrow\quad G\cong A_5,\\
\mathcal C_{\rm nonsol}:\quad
& G\text{ finite nonsolvable},\ \gamma(G)=9/2
  &&\centernot\Longrightarrow\quad G\cong A_5,\\
\mathcal C_{\rm rad}:\quad
& G=A_5\times N,\ N\text{ nilpotent},\ \Rad(G)=1
  &&\Longrightarrow\quad G\cong A_5,\\
\mathcal C_{\rm cen}:\quad
& G=A_5\times N,\ N\text{ nilpotent},\ Z(G)=1
  &&\Longrightarrow\quad G\cong A_5.
\end{aligned}
\end{equation}
Here the first line is the simple-group theorem, the second line is Theorem \ref{thm:main-counterexample}, and the last two lines are the within-family consequences of Proposition \ref{prop:family-repairs}.  The second line is the new contribution: it is the direct contradiction to the nonsolvable conjectural extension.  The last two lines only clarify that the contradiction is produced by a nilpotent radical, not by the almost simple component.

\section{Exact arithmetic of the displayed certificates}

For completeness, and to make the rational verification reproducible, we expand the four certificates of Proposition \ref{prop:more-certificates} as numerator partitions of \(1120\) over the common denominator \(2520\).  Let
\begin{equation}\label{eq:common-denominator-D}
        D=2520=2^3\cdot3^2\cdot5\cdot7.
\end{equation}
For a prime set \(Q\), the identity \(\Theta(Q)=4/9\) is equivalent to
\begin{equation}\label{eq:numerator-certificate-general}
        \sum_{q\in Q}\frac{D}{q+1}=\frac{4D}{9}=1120.
\end{equation}
The four certificates are therefore the four exact partitions
\begin{equation}\label{eq:partition1}
        315+210+180+140+126+84+35+30=1120,
\end{equation}
\begin{equation}\label{eq:partition2}
        315+210+180+140+126+105+30+14=1120,
\end{equation}
\begin{equation}\label{eq:partition3}
        315+210+180+140+126+84+60+5=1120,
\end{equation}
\begin{equation}\label{eq:partition4}
        315+210+180+140+126+105+42+2=1120.
\end{equation}
Each summand in \eqref{eq:partition1}--\eqref{eq:partition4} is of the form \(D/(q+1)\) for the corresponding prime \(q\).  Dividing by \(D\) gives the reciprocal identities, and Theorem \ref{thm:certificate} converts each identity into a finite group.

This last step is purely symbolic:
\begin{equation}\label{eq:symbolic-chain}
\begin{array}{ccccc}
\displaystyle \sum_{q\in Q}\frac{D}{q+1}=1120
&\Longleftrightarrow&
\displaystyle \Theta(Q)=\frac{1120}{2520}=\frac49
&\Longleftrightarrow&
\displaystyle \gammainv(G(Q))=\frac{73}{18}+\frac49=\frac92.
\end{array}
\end{equation}
No numerical approximation, element enumeration, or classification-theoretic input is used in \eqref{eq:symbolic-chain}.

\section{Conclusion}

The conjectural characterization \eqref{eq:conjectural-statement} fails for a structural reason.  The Sylow-integral invariant \(\gamma\) depends simultaneously on the number of Sylow subgroups and on the orders of Sylow subgroups.  Direct product with a nilpotent factor keeps the Sylow numbers of the nonabelian factor but changes the denominator attached to any shared prime; at the same time, every new normal Sylow subgroup contributes a positive reciprocal.  For \(A_5\), the central factor \(\C_2\) creates the exact deficit
\[
        1-\frac59=\frac49,
\]
and the prime set \(\{7,11,13,17,19,29,71,83\}\) supplies the exact reciprocal sum
\[
        \frac18+\frac1{12}+\frac1{14}+\frac1{18}+\frac1{20}+\frac1{30}+\frac1{72}+\frac1{84}=\frac49.
\]
The group
\[
        A_5\times\C_2\times\C_7\times\C_{11}\times\C_{13}\times\C_{17}\times\C_{19}\times\C_{29}\times\C_{71}\times\C_{83}
\]
therefore has \(\gamma=9/2\) while being nonsolvable and not isomorphic to \(A_5\).  The stronger compensation theorem, Theorem \ref{thm:A5-compensation}, gives an exact criterion for all direct products \(A_5\times N\) with \(N\) nilpotent and shows that the counterexample is a manifestation of a general balancing equation.

\bmhead{Statements and Declarations}

\noindent\textbf{Funding.} No funding was received for this work / insert funding information here.

\noindent\textbf{Competing interests.} The authors declare that they have no competing interests.

\noindent\textbf{Data availability.} No datasets were generated or analysed during the current study.

\noindent\textbf{Author contributions.}
Y.Z. conceived the main problem, developed the principal arguments, and drafted the manuscript. Y.Y. contributed to the verification and refinement of the proofs, improved the exposition, and assisted with the preparation of the final version. Both authors discussed the results, reviewed the manuscript, and approved the final version for submission.

\section*{Declaration of Generative AI and AI-Assisted Technologies in the Writing Process}
During the preparation of this work, the authors used DeepSeek to build a specialized agent for solving mathematical problems, which was employed to generate an initial proof of the main theorem. After using this tool, the authors reviewed and edited the content as needed and take full responsibility for the content of the published article.

\end{document}